\documentclass[a4paper,10pt]{amsart}

\usepackage[english]{babel}
\usepackage{amsmath,amssymb,amsthm}

\newcommand{\metric}[2]{\ensuremath{\langle #1, #2\rangle}}  
\newcommand{\nks}{\ensuremath{S^3\times S^3}}   
\renewcommand{\epsilon}{\varepsilon}            
\newcommand{\e}{\epsilon}                       
\renewcommand{\aa}{\ensuremath{\metric{\alpha}{\alpha}}}  
\newcommand{\ab}{\ensuremath{\metric{\alpha}{\beta}}}     
\newcommand{\bb}{\ensuremath{\metric{\beta}{\beta}}}      
\newcommand{\axb}{\ensuremath{\alpha\times\beta}}         
\renewcommand{\d}{\ensuremath{\partial}}
\newcommand{\db}{\ensuremath{\bar{\partial}}}
\newcommand{\dph}{\ensuremath{\d\phi}}
\newcommand{\Z}{\ensuremath{\mathbb{Z}}}
\newcommand{\R}{\ensuremath{\mathbb{R}}}
\newcommand{\C}{\ensuremath{\mathbb{C}}}
\renewcommand{\H}{\ensuremath{\mathbb{H}}}
\newcommand{\F}{\ensuremath{\mathcal{F}}}
\newcommand{\fb}{\ensuremath{\bar{f}}}
\newcommand{\fe}{\ensuremath{f^\epsilon}}
\newcommand{\x}{\ensuremath{(\mu - 2\e i\dph)}}
\newcommand{\xx}{\ensuremath{\mu - 2\e i\dph}}
\renewcommand{\ae}{\ensuremath{\alpha^\e}}
\newcommand{\be}{\ensuremath{\beta^\e}}

\DeclareMathOperator{\im}{Im}    
\DeclareMathOperator{\csch}{csch}    
\DeclareMathOperator{\sech}{sech}    
\DeclareMathOperator{\arccot}{arccot}    

\newtheorem{theorem}{Theorem}[section]   
\newtheorem*{theorem*}{Theorem}          
\newtheorem{lemma}[theorem]{Lemma}
\newtheorem{proposition}[theorem]{Proposition}
\newtheorem{corollary}[theorem]{Corollary}

\theoremstyle{definition}

\newtheorem{remark}[theorem]{Remark}
\newtheorem{remarks}[theorem]{Remarks}

\title{Sequences of harmonic maps in the 3-sphere}
\author{Bart~Dioos \and Joeri Van der Veken \and Luc~Vrancken}
\address{Bart Dioos, Joeri Van der Veken, Luc Vrancken,
 KU\ Leuven, Departement Wiskunde, 
 Celestijnenlaan 200B,
 3001 Leuven, Belgium}
\email{bart.dioos@wis.kuleuven.be}
\email{joeri.vanderveken@wis.kuleuven.be}
\address{Luc Vrancken, LAMAV, 
 Universit\'e de Valenciennes, 
 Campus du Mont Houy, 59313 
 Valenciennes Cedex 9, France}
\email{luc.vrancken@univ-valenciennes.fr}
\thanks{The first author was partially supported by the Belgian Interuniversity Attraction Pole P07/18 (Dygest).}

\begin{document}

\begin{abstract}
We define two transforms between non-conformal harmonic maps from a surface into the 3-sphere. 
With these transforms one can construct, from one such harmonic map, a sequence of harmonic maps. 
We show that there is a correspondence between non-conformal harmonic maps into the 3-sphere, $H$-surfaces in Euclidean 3-space
and  almost complex surfaces in the nearly K\"ahler manifold~$\nks$.
As a consequence we can construct sequences of $H$-surfaces and almost complex surfaces.
\end{abstract}

\maketitle

\section{Introduction}
\label{sec:intro}

Consider a map~$f\colon S \to S^3$ from a Riemann surface~$S$ into the unit~$3$-sphere. 
The map~$f$ is harmonic if it satisfies~the equation~$\Delta f + |df|^2 f=0$ where $\Delta$
is the Laplacian on the surface~$S$ (\cite{helein}). 
The map~$f$ is conformal if it preserves the conformal structure on~$S$, that is,~$\metric{\d f}{\d f}=0$
where~$\d$ stands for~$\tfrac{\d}{\d z}$.
In this case,~$f$ is a minimal immersion of the surface in the $3$-sphere. 
For a recent and broad survey on harmonic maps, the reader is referred to~\cite{heleinwood}.
In this article we will almost always assume that~$f$ is not conformal: $\metric{\d f}{\d f}\neq 0$.
We will show that to a non-conformal harmonic map from~$S$ to~$S^3$ one can associate two new
maps from~$S$ to~$S^3$ which are also non-conformal and harmonic. In fact one can define a sequence
$\{f^p\mid p \in \Z\}$ of non-conformal harmonic maps from~$S$ into~$S^3$ where~$f^0= f$.
These will be the main results of Section~\ref{sec:transforms}.

The transforms of harmonic maps were inspired by the work~\cite{boltonvrancken}.
In that article Bolton and the last author described transforms to minimal
surfaces in~$S^5$ with non-circular ellipse of curvature. 
Antic and the last author~\cite{anticvrancken} generalized these transforms for superconformal minimal surfaces in odd-dimensional
spheres~$S^{2n+1}$ whose $(n-2)$ higher-order ellipses of curvature are circles. 
These transforms are natural generalisations of the polar construction for superconformal minimal surfaces
in odd-dimensional spheres (see~\cite{bpw} and for surfaces in~$S^3$ see~\cite{lawson}).

In Section~2 we first give some preliminaries on quaternions.
Quaternions turn out to be very convenient for describing surfaces in the 3-sphere.
A study of surfaces in spheres using a quaternionic language can be found in~\cite{burstall}.
In this reference however conformal maps are studied whereas in this article we will mostly consider non-conformal maps.
Next we will define an adapted frame for non-conformal harmonic maps and
derive the moving frame equations and the compatibilty conditions for this frame. 

In Section~\ref{sec:harmsequences} we will address the following problem.
Given a non-conformal harmonic map~$f$ from a surface into the 3-sphere we can transform it into a new non-conformal harmonic map.
When are a non-conformal map~$f$ and its transformed map equal up to an isometry of~$S^3$? 
The maps satisfying this property are certain parametrizations of Clifford tori.
Since the transforms are generalizations of Lawson's polar construction for minimal surfaces, this question is analogue to the question
when a minimal surface is congruent to its polar surface. 
The only minimal surface~in~$S^3$ congruent to its polar is the Clifford torus~$S^1(\tfrac{1}{\sqrt{2}})\times S^1(\tfrac{1}{\sqrt{2}})$.
In order to answer our question we will prove an existence and uniqueness theorem (Proposition~\ref{prop:bonnet})
that very much resembles the classical existence and uniqueness theorem of Bonnet. Using the
Bonnet-type theorem and a technical lemma, we can answer our question in Theorem~\ref{thm:equiv}.

The original motivation of the authors to investigate these transforms of harmonic maps is the 
study of almost complex surfaces in the nearly K\"ahler manifold~$\nks$.
In the paper~\cite{bddv}, Bolton, Dillen, Dioos and Vrancken found a correspondence between almost complex surfaces in~$\nks$
and $H$-surfaces~$X$ in~$\R^3$, that is, surfaces that satisfy the Wente~$H$-equation~$X_{xx}+X_{yy}=-\tfrac{4}{\sqrt{3}}X_x \times X_y$.
Moreover, on such an almost complex surface there exists a quadratic holomorphic differential. 
In Section~\ref{sec:acsurfaces} of the present paper we will discuss the relation between harmonic maps to~$S^3$,
$H$-surfaces in~$\R^3$ and almost complex surfaces in~$\nks$.
In Theorem~\ref{thm:acsurface} we prove that almost complex surfaces in~$\nks$
correspond to harmonic maps in the 3-sphere, and vice versa. 
This correspondence follows quickly from the results in~\cite{bddv} we just mentioned and a non-conformal analogue of
Lawson's correspondence Theorem (Proposition~\ref{prop:lawson}).
As a corollary one can associate a whole sequence of such surfaces to one given almost complex surface~in~$\nks$ with non-vanishing differential.

\section{Harmonic maps to~$S^3$}
\label{sec:harmmaps}

The ring of quaternions~$\H$ can be identified with the vector space~$\R^4$. 
If quaternions are written as real linear combinations of the basis elements $1$,~$e_1$,~$e_2$ and~$e_3$, then 
the quaternion multiplication is determined completely by the identities
\[
   e_1^2 = e_2^2 = e_3^2 = e_1 e_2 e_3 = -1.
\]
A quaternion that is a linear combination of~$e_1$,~$e_2$ and~$e_3$ is called an imaginary quaternion.
The set of imaginary quaternions~$\im \H$ can be identified with the Euclidean space~$\R^3$.
The product of two imaginary quaternions~$\alpha$ and~$\beta$ is given by
\begin{equation}
\label{eq:improd}
   \alpha \beta = -\metric{\alpha}{\beta} + \alpha \times \beta
\end{equation}
where~$\metric{\,}{\,}$ is the Euclidean inner product and~$\times$ is the usual vector product on~$\R^3$.

The quaternions are very useful to describe the 3-sphere and its tangent spaces.
The 3-sphere~$S^3$ is the set of unit quaternions~$\{p \in \H \mid \|p\| = 1\}$. 
One can prove that~$\metric{uv}{uw}=\metric{u}{u}\metric{v}{w}$ for all quaternions~$u$, $v$ and~$w$. 
Therefore $p\alpha$ is orthogonal to~$p$ for every imaginary quaternion~$\alpha$
and the tangent space at~$p$ is 
\begin{equation}
 \label{eq:tang}
 T_p S^3 = \{p\alpha \mid \alpha \in \im \H\}.
\end{equation}
On a surface we will use complex coordinates, so in order to describe the complexified tangent vectors
we will need the complexified quaternions~$\H\otimes \C= \H \oplus i \H$.
The element~$i$ must be distinguished from~$e_1 \in \H$. The complex bilinear extension of the
Euclidean metric and vector product will also be denoted by~$\metric{\,}{\,}$ and~$\times$.
The product of two complexified quaternions~$p_1+ip_2$ and~$q_1+iq_2$ is
\begin{equation*}
 (p_1+ip_2)(q_1+iq_2) = (p_1 q_1 - p_2 q_2) +i(p_1 q_2+p_2 q_1). 
\end{equation*}
The conjugate of a complexified quaternion~$p_1+ip_2$ is equal to~
\[ \overline{p_1+ip_2}=p_1 - ip_2 \] 
and is denoted with a bar. 
In contrast, the conjugate of a quaternion~$p$ is written as
\[
   p^* = (a + b e_1 + c e_2 + d e_3)^* = a - b e_1 - c e_2 - d e_3.
\]
So~$\,\bar{ }\,$ means conjugation with respect to the imaginary~$i$ and~${}^*$ is conjugation
with respect to the three imaginaries~$e_1$,~$e_2$,~$e_3$.

Now consider a harmonic map~$f\colon S\to S^3\subset \H$ from a Riemann surface~$S$ into the $3$-sphere~$S^3$. 
Choose a local complex coordinate~$z=x+iy$ on~$S$. 
Denote~$\tfrac{\partial}{\partial z}$ and~$\tfrac{\partial}{\partial \bar{z}}$ by~$\d$ and~$\db$ respectively. 
Similarly the derivatives of~$f$ with respect to the real coordinates~$x$ and~$y$ will be written as~$f_x$ and~$f_y$ respectively. 
We introduce the~$\H\otimes \C$-valued function
\[
    f_1 = \d f.
\]
Since~$\metric{f}{f}=1$, it follows that~$\metric{f}{f_1}=0$. 

Harmonicity means that~$\d\db f = -|\d f|^2 f$.
If~$f$ is in addition a conformal map,~$f$ is a minimal isometric immersion of~$S$ in~$S^3$. 
We assume that~$f$ is not conformal.
By the harmonicity of~$f$, the function~$\metric{f_1}{f_1}$ is holomorphic. 
Since~$f$ is assumed to be non-conformal, $\metric{f_1}{f_1}$ is non-zero. 
Therefore there exists a complex coordinate~$z$ such that~$\metric{f_1}{f_1}=-1$. 
We will call such a coordinate an \emph{adapted complex coordinate for~$f$}.

By \eqref{eq:tang} there exist functions~$\alpha$ and~$\beta$ with values in~$\im\H$ such that~$f_x = f\alpha$ and~$f_y = f\beta$.
This means that~$f_1=\tfrac{1}{2}f(\alpha-i \beta)$. 
It follows from $\metric{f_1}{f_1}=-1$ that
\begin{align*}
   &\aa - \bb = - 4,  &  &\ab = 0.
\end{align*}
Hence there is a non-negative smooth function~$\phi$ such that
\begin{align}
 \label{eq:aabb}
 |\alpha| &= 2\sinh \phi, &    |\beta| &= 2\cosh \phi.
\end{align}
Note that~$\beta$ is nowhere vanishing by equation~\eqref{eq:aabb}, but~$\alpha$ can be zero. At points were~$\alpha$
is not zero the vectors~$f_1$ and~$\fb_1$ are linearly independent and~$\phi$ is positive.
In the following we will tacitly assume that we are working on the open subset~$U$ of~$S$ where~$\alpha\neq 0$.
On this set $\phi$ is positive and the image~$f(U)$ is a surface in the~3-sphere.
At a point of~$U$ define~$N$ to be the real unit vector in the direction of~$f(\axb$). Then~$N$
is orthogonal to~$\{f,f_1,\fb_1\}$ and
\[
   |f(\axb)|^2 = \aa \bb - \ab^2 = 4\sinh^2 2\phi,
\] 
hence~$f(\axb) = \pm 2\sinh 2\phi\, N$. 
For definiteness we take~$N$ such that the orthogonal frame~$\{f, f_x, f_y, N\}$ 
in~$\R^4$ is positively oriented, so~$N=\tfrac{1}{2}\csch 2\phi\,f(\axb)$.
We have now defined a complex moving frame~$\F=\{f,f_1,\fb_1,N\}$ for~$f$ on the set~$U$.
The matrix~$A$ of complex inner products of the vectors of~$\F$ is
\begin{equation}
\label{eq:A}
A=
 \begin{pmatrix}
 1 & 0 & 0             & 0 \\
 0 & -1 & \cosh 2\phi  & 0 \\
 0 & \cosh 2\phi  & -1 & 0 \\
 0 & 0 & 0 & 1
 \end{pmatrix}.
\end{equation}

Now we give the moving frame equations for~$\F$. Write~$\mu = \metric{\d f_1}{N}$. 
An easy calculation using~\eqref{eq:A} gives the equations in terms of~$\phi$ and~$\mu$.
\begin{align}
\label{eq:mfeqs}
\begin{split}
 \d f     &= f_1,\\
 \d f_1   &= f + 2\dph (\coth 2\phi \,f_1 + \csch 2\phi\,\fb_1) + \mu\,N, \\
 \d \fb_1 &= -\cosh 2\phi\,f, \\
 \d N     &= -\mu \csch 2\phi (\csch 2\phi\,f_1 + \coth 2\phi\, \fb_1).
\end{split}
\end{align}
The corresponding~$\db$-equations can be found by simply taking the conjugates of the above ones. 
The compatibility conditions~$\d\db \F = \db \d \F$ for the frame~$\F$ give
\begin{align}
\label{eq:compcond}
\begin{split}
2 \d\db \phi &= -\sinh 2\phi + |\mu|^2 \csch 2\phi,\\
\db \mu &= -2 \bar{\mu}\dph \csch 2\phi.
\end{split}
\end{align}
The complex function~$\mu$ measures the rate at which the image of~$f$ is pulling away from the great 2-sphere tangent to
the image of~$f$. If~$f$ is a map into a great 2-sphere, then~$\mu$ vanishes
and the above compatibility condition for~$\phi$ becomes the sinh-Gordon equation. 
\begin{remarks} 
\label{rem:remarksphere}
\begin{enumerate}
\item
A non-conformal harmonic map~$f$ from a compact surface~$S$ to the 2-sphere always must have a singular point.
Indeed, suppose such a map does not have singular points, then~$\phi$ is a positive function on a compact surface
satisfying the sinh-Gordon equation. But then the maximum principle gives a contradiction.
\item
Recall that by Hopf's lemma there are no non-conformal harmonic maps from the 2-sphere to the 3-sphere.
\end{enumerate}
\end{remarks}

\section{The transforms}
\label{sec:transforms}

Let~$f\colon S\to S^3$ be a non-conformal harmonic map. 
In this section we will first show how to associate to~$f$ two new non-conformal harmonic maps~$f^+$ and~$f^-$ from~$S$
into~$S^3$. Moreover, if~$z$ is an adapted complex coordinate for the map~$f$, then
it also is an adapted coordinate for~$f^+$ and~$f^-$. Then we will show that the two transformations are each others inverse
in the sense that~$(f^+)^- = (f^-)^+ = f$. Consequently we can associate a sequence
$\{f^p\mid p \in \Z\}$ of non-conformal harmonic maps from~$S$ into~$S^3$ to the map~$f^0= f$.

Fix a point $p \in S^3$ and consider the vectors
\begin{align}
\label{eq:vectors}
  \pm\sin\theta \frac{f\beta}{|f\beta|} + \cos\theta N
\end{align}
in $T_{f(p)}S^3$, where~$\theta$ is chosen such that~$\cos\theta = |\alpha|/|\beta|=\tanh\phi$
and~$\sin \theta = \sech\phi$. The ellipse~$E$ with~$f\alpha$ and~$f\beta$ as minor and major semi-axes is the image of 
a circle in the tangent plane to~$S$ at~$p$ under~$df$.
The cosine~$\cos\theta$ is the ratio between the lengths of the minor and major axes of this ellipse and
is a measure for its eccentricity as well as for the non-conformality of~$f$.
The vectors above now have the following geometrical meaning. 
Let~$R_\theta$ be the rotation of $T_{f(p)}S^3$ about the minor axis of $E$ through the angle~$\theta$.
Then the orthogonal projection of the rotated ellipse~$R_\theta(E)$ onto the plane containing~$E$ is a circle. Of course, the same
holds for the rotation $R_{-\theta} = R_{\theta}^{-1}$.
The vectors above are the images of the unit normal~$N$ under the rotations~$R_\theta$ and $R_{-\theta}$.  

We can rewrite the vectors in \eqref{eq:vectors} as
\begin{align}
\label{eq:def}
 \begin{split}
  f^+ &= \phantom{-}\frac{i}{2}\sech^2\phi (f_1-\fb_1) + \tanh\phi\, N,\\
  f^- &= -\frac{i}{2}\sech^2\phi (f_1-\fb_1) + \tanh\phi\, N,
 \end{split}
\end{align}
where we have used~$f\beta=i(f_1-\fb_1)$ and~$|f\beta|=2 \cosh\phi$.
By varying the point~$p$, we can look at~$f^+$ and~$f^-$ as maps from~$S$ to $S^3$ again and we call them the~\emph{$(+)$transform} and~\emph{$(-)$transform} of $f$ respectively.
If~$f$ were conformal, that is if~$|\alpha|=|\beta|$ every\-where, the expressions~\eqref{eq:vectors} would still make sense. In fact, we get~$\theta=0$ and recover
the polar surface of the minimal surface $f$, see~\cite{lawson}. From now on we restrict to the non-conformal case and we denote the transforms \eqref{eq:def} by~$\fe$ where~$\e$ is $1$ or~$-1$. 
All objects and functions related to~$f^\epsilon$ will be denoted with a superscript~$\epsilon$.
For instance,~$\alpha^\e$ and~$\beta^\e$ are the functions such that~$\d \fe = \tfrac{1}{2}\fe (\alpha^\e - i\beta^\e)$
and~$\phi^\e$ is the non-negative smooth function that satisfies the~$\e$-analogue of~\eqref{eq:aabb}.

\begin{remark}
In~\cite[p.~64-65]{helein} H\'elein describes how to associate to a harmonic map
from a surface to~$S^2$ two new harmonic maps from the surface into the 3-sphere. 
These harmonic maps are conformal, so they are different from the transforms we describe here. 
\end{remark}

\begin{theorem}
 \label{thm:harmonic}
Let~$f\colon S\to S^3$ be a non-conformal harmonic map from a Riemann surface~$S$ into the 3-sphere.
Then the tranforms~$f^+$ and~$f^-$ are also non-conformal harmonic maps from~$S$ to~$S^3$.
Furthermore, an adapted complex coordinate for~$f$ is also an adapted complex coordinate for~$f^+$
and~$f^-$.
\end{theorem}
\begin{proof}
First we define~$\fe_1= \d \fe$. 
Using the definition~\eqref{eq:def} of~$\fe$ and the moving frame equations~\eqref{eq:mfeqs} for~$\F$
one gets
\begin{equation}
\label{eq:fe1}
 \fe_1 = \e i f -\frac{1}{2}\x \sech^2\phi (\csch2\phi\,f_1 +\coth2\phi\,\fb_1 - \e i \, N).
\end{equation}
A computation using the inner products~\eqref{eq:A} then gives
\begin{align}
 &\metric{\fe_1}{\fe_1} = -1, \notag\\
 &|\fe_1|^2 = 1 + \frac{1}{2}|\xx |^2\sech^2\phi. \label{eq:normfe1}
\end{align}
So the coordinate~$z$ is an adapted complex coordinate for~$f^\epsilon$ as well.

Next we will show that~$f^\epsilon$ is a non-conformal harmonic map. 
We have just showed that~$\metric{\fe_1}{\fe_1}$ is non-zero, so~$\fe_1$ clearly is non-conformal.
We still have to check that~$\d\db\fe$ is a multiple of~$\fe$. 
An easy calculation using~\eqref{eq:A} gives
\begin{align}
\label{eq:someinprod}
\begin{split}
  &\metric{\fe}{f}=0, \\     
  &\metric{\fe}{f_1}=-\e i, \\ 
  &\metric{\fe}{\fb_1}=\e i, \\
  &\metric{\fe}{N} =\tanh \phi.
\end{split}
\end{align}
Equation~\eqref{eq:fe1} gives
\begin{align}
\label{eq:someiptwo}
\begin{split}
 \metric{\fe_1}{f}&= \e i, \\
 \metric{\fe_1}{f_1}&=-\tanh\phi \x,  \\
 \metric{\fe_1}{\fb_1}&=0, \\
 \metric{\fe_1}{N} &= \frac{\e i}{2} \sech^2\phi \x.
\end{split}
\end{align}
Now we can calculate the inner products of~$\d\db\fe$ with the frame vectors of~$\F$.
The first inner product is
\[
  \metric{\d\db\fe}{f}= \db\metric{\fe_1}{f}-\metric{\fe_1}{\fb_1} = 0.
\]
With a similar calculation using the compatibility conditions~\eqref{eq:compcond} and equation~\eqref{eq:normfe1}
the other inner products become
\begin{align*}
 \metric{\d\db\fe}{f_1} &= \e i |\fe_1|^2,\\
 \metric{\d\db\fe}{\fb_1} &= -\e i |\fe_1|^2,\\
 \metric{\d\db\fe}{N} &= -\tanh\phi |\fe_1|^2.
\end{align*}
Comparing with~\eqref{eq:someinprod} we have now shown that~$\d\db\fe=-|\fe_1|^2\fe$ and thus~$\fe$ is an harmonic map into the 3-sphere.
\end{proof}

\begin{lemma}
 \label{lem:phie}
The function~$\phi^\e$ is positive on an open dense subset~$U^\e$ of~$S$.
\end{lemma}
\begin{proof}
From the expressions for~$\metric{\fe_1}{\fe_1}$ and~$|\fe_1|^2$
it follows that~$\metric{\alpha^\e}{\alpha^\e}=|\mu-2\e i \dph|^2 \sech^2\phi$.
So~$\phi^\e$ is non-negative and is zero if and only if~$\mu=2\e i \dph$.
The compatibility conditions~\eqref{eq:compcond} for the moving frame~$\F$ give
\[
 \db \x = \e i \bigl(\sinh 2\phi - \bar{\mu}\csch 2\phi \x\bigr).
\] 
If $\mu=2\e i \dph$ on an open set, the equation gives~$\sinh 2\phi=0$, which means~$\phi=0$.
This gives a contradiction.
\end{proof}

By Lemma~\ref{lem:phie} the normal~$N^\e$ and the moving frame~$\F^\e = \{\fe, \fe_1, \fb^\e_1, N^\e\}$ 
of the new harmonic map~$\fe$ are well-defined on an open dense set~$U^\e$. 
The $\e$-analogues of the moving frame equations and compatibility conditions for~$\F^\e$ hold on~$U^\e$.

Next we prove that the~$(+)$transform and $(-)$transform are mutual inverses. Before doing so, we 
calculate the functions~$\alpha^\e$, $\beta^\e$ and~$\alpha^\e\times \beta^\e$ 
in terms of~$\alpha$, $\beta$, $\phi$ and~$\mu$.
\begin{lemma}
 \label{lem:calc}
If $\mu=\mu_1 + i \mu_2$ and~$\phi_x= \tfrac{\dph}{\d x}$ and~$\phi_x= \tfrac{\dph}{\d y}$,
then
\begin{align*}
\alpha^\e &= (\mu_2-\e \phi_x)\csch 2\phi\, \alpha  \\
            &\quad +\frac{1}{2}(\mu_1-\e \phi_y)\tanh\phi\sech^2\phi\bigl(\beta - \frac{\e}{2}\csch^2\phi\,\axb \bigr),\\
\beta^\e &=  (\mu_1-\e \phi_y)\csch 2\phi\, \alpha 
            +\sech^2\phi \bigl( 1-\tfrac{1}{2}(\mu_2-\e\phi_x)\tanh\phi\bigr)\beta\\ 
                       & \quad  +\frac{\e}{2}\sech^2\phi\bigl(1+(\mu_2-\e\phi_x)\csch 2\phi\bigr)\axb ,\\
\alpha^\e \times\beta^\e &= 2\e\csch 2\phi(\mu_1-\e \phi_y)\,\alpha
       -\frac{\e}{2}\sech^4\phi \bigl(|\mu|^2+(\mu_2-\e\phi_x)\sinh2\phi\bigr)\beta \\
       &\quad -\frac{1}{4}\sech^4\phi \bigl(|\mu|^2+2(\mu_2-\e\phi_x)\coth\phi\bigr)\axb.           
\end{align*}
\end{lemma}
\begin{proof}
The~$\e$-transform~$\fe$ is a map to~$S^3$, so~$\fe(\fe)^* = 1$. Hence~$\ae-i\be= 2(\fe)^* \fe_1$.
The definition of the transform~$\fe$ and expression~\eqref{eq:fe1} for~$\fe_1$ give 
\begin{align*}
\ae-i\be &= \bigl(\e i\sech^2\phi (f_1-\fb_1)^* + 2\tanh\phi\, N^* \bigr) \cdot \\
        &\qquad \Bigl(\e i f -\frac{1}{2}\x \sech^2\phi (\csch 2\phi\,f_1 +\coth 2\phi\,\fb_1 - \e i \, N\Bigr)
\end{align*}
The equality~$(p\alpha)^* = -\alpha p^*$ holds~for every quaternion~$p$ and every imaginary quaternion~$\alpha$.
This equality and~\eqref{eq:aabb} give
\begin{align*}
 N^*f&= -\frac{1}{2}\csch 2\phi\,\axb,                    &(f_1-\fb_1)^* f&= i \beta, \\
 N^*f_1&= -\frac{1}{2}\bigl(\tanh\phi\,\beta+i\coth\phi\,\alpha\bigr),   &(f_1-\fb_1)^* f_1 &= -2\cosh^2\phi - \frac{i}{2}\axb,\\
 N^*\fb_1&= -\frac{1}{2}\bigl(\tanh\phi\,\beta-i\coth\phi\,\alpha\bigr), &(f_1-\fb_1)^* N &= i \coth\phi\, \alpha.
\end{align*}
Substitution of these equations give 
\begin{align*}
\ae - i \be &= -i\sech^2\phi \Bigl(\beta + \frac{\e}{2}\axb\Bigr)
                     -\x\bigl(i\csch 2\phi\,\alpha \bigr.\\
                   & \quad \bigl. -\tfrac{1}{2}\tanh\phi\sech^2\phi\,\beta 
                                      +\tfrac{\e}{4}\coth\phi\sech^4\phi\,\axb  \bigr).
\end{align*}
After expanding the product in the last term and simplifying some hyperbolic trigonometric identities,
one obtains the expressions for~$\alpha^\e$ and~$\beta^\e$. The expression for~$\alpha^\e\times\beta^\e$
can be calculated by using the formula~$u\times(v\times w) = \metric{u}{w}v-\metric{u}{v}w$ and the inproducts~\eqref{eq:aabb}.
\end{proof}

\begin{theorem}
\label{thm:inverses}
Let~$f\colon S\to S^3$ be a non-conformal harmonic map. Then the $(+)$transform and $(-)$transform of~$f$ 
are mutual inverses in the sense that
\[ (f^+)^- = (f^-)^+ = f.\]
\end{theorem}
\begin{proof}
Consider a non-conformal harmonic map~$f\colon S\to S^3$.  
Let~$\e$ and~$\tilde \e$ be such that~$\e  \tilde\e = -1$.
Note that 
\[
 \fe =  \frac{1}{2}\sech^2\phi f\Bigl(\e\beta+\frac{1}{2}\axb\Bigr).
\]
Therefore
\begin{align*}
 {(\fe)}^{\tilde \e}&= \frac{1}{2}\sech^2\phi^\e \fe\Bigl(\tilde\e\beta^\e+\frac{1}{2}\alpha^\e \times\beta^\e\Bigr)\\
                    &= \frac{1}{4}\sech^2\phi\sech^2\phi^\e 
                     f\Bigl(\e\beta+\frac{1}{2}\axb\Bigr)\Bigl(\tilde\e\beta^\e+\frac{1}{2}\alpha^\e \times\beta^\e\Bigr).
\end{align*}
Lemma~\ref{lem:calc} and a good amount of arithmetic then yield
\begin{equation}
{(\fe)}^{\tilde \e} = \frac{1}{4}\sech^2\phi\sech^2\phi^\e f\big(|\xx|^2+4\cosh^2\phi\bigr) = f.
\end{equation}
So the $(+)$transform and $(-)$transform are each others inverses.
\end{proof}

Theorem~\ref{thm:inverses} allows us to associate to a non-conformal harmonic map~$f\colon S \to S^3$
a sequence $\{f^p\mid p\in\Z \}$ of such harmonic maps by defining $f^0 = f$ and, 
for every integer~$p$, $f^{p+1}=(f^p)^+$ and $f^{p-1}=(f^p)^-$.
Moreover, if~$z$ is an adapted complex coordinate for one of the maps~$f^p$ in the sequence, then
it is an adapted complex coordinate for every map in the sequence. In the rest of the article 
we will refer to this sequence as the \emph{sequence associated to the map~$f$}. 

\section{Sequences of harmonic maps}
\label{sec:harmsequences}

In this section we answer the following question:
when is a non-conformal harmonic map~$f$ from a surface to the 3-sphere the same as its~$\e$-transform~$\fe$?
Two maps are considered to be the same if they are equal up to an isometry of the 3-sphere.
Note that in this case all maps in the sequence associated to~$f$ are congruent by Theorem~\ref{thm:inverses}.
Proposition~\ref{prop:bonnet} gives a criterium for this. It tells us that two non-conformal harmonic
maps are equal, up to an isometry of~$S^3$, if their respective functions~$\phi$ and~$\mu$ 
as defined in Section~\ref{sec:harmmaps} agree. Proposition~\ref{prop:bonnet} is an analogue of the classical
existence and uniqueness theorem of Bonnet.
Next we find two relations between the functions~$\phi^\e$ and~$\mu^\e$ of the $\e$-transform~$\fe$
and~$\phi$ and~$\mu$ of the original map~$f$. With the use of the Bonnet-type theorem and the 
relations between~$\phi, \mu$ and~$\phi^\e, \mu^\e$, we determine
all the non-conformal maps that are congruent to their~$\e$-transforms.

\begin{proposition}
\label{prop:bonnet}
Let~$\phi$ be a real positive smooth function and~$\mu$ a complex function on a simply connected surface~$S$
satisfying the differential equations~\eqref{eq:compcond}. 
Then there exists a non-conformal harmonic map~$f\colon S\to S^3$ such that all the
inner products of the frame~$\F=\{f,f_1,\fb_1,N\}$ are given by~\eqref{eq:A} and~$\metric{\d f_1}{N}=\mu$. 
Furthermore such a map~$f$ is unique up to isometries of~$S^3$:
if~$g$ is another map satisfying the above conditions, then~$g=R\circ f$ where~$R\in O(4)$.
\end{proposition}
\begin{proof}
(Existence.) Fix a point~$p$ in~$S$. 
We consider the moving frame equations~\eqref{eq:mfeqs} as a system of first order
differential equations with the $\R^4$-valued functions~$f$,~$f_1$, $\fb_1$ and~$N$ as variables.
Since the integrability conditions are assumed to hold, 
there exists a unique solution~$f$, $f_1$, $\fb_1$ and~$N$ on~$S$ such that
the ten inner products~$\metric{f}{f}$, $\metric{f}{f_1}$, \ldots, $\metric{N}{N}$ are given by the matrix~\eqref{eq:A}
at the point~$p$. 

The inner products satisfy the system of twenty linear differential equations
\begin{align*}
 \d \metric{f}{f} &= F_1(\metric{f}{f},\ldots,\metric{N}{N}, \phi,\mu),
        & \db \metric{f}{f} &= F'_1(\metric{f}{f},\ldots,\metric{N}{N}, \phi,\mu),\\
   &   \vdots & \vdots& \\
  \d \metric{N}{N} &= F_{10}(\metric{f}{f},\ldots,\metric{N}{N}, \phi,\mu),  
        & \db \metric{N}{N} &= F'_{10}(\metric{f}{f},\ldots,\metric{N}{N}, \phi,\mu).    
\end{align*} 
The functions~$F_1,\ldots, F'_{10}$ are linear with respect to the inner products and depend on~$\phi$ and~$\mu$. 
It is a straightforward task to check that the functions in the entries of the matrix~$A$ also satisfy this system of equations.
Since both solutions agree at the point~$p$, they must be the same at all points of~$S$.

Finally note that the two differential equations~$\d f= f_1$, $\db f = \fb_1$ are integrable,
so there exists a smooth map~$f$ from~$S$ to~$\R^4$. 
The inner products $\metric{f}{f_1}$ and~$\metric{f}{\fb_1}$ are zero, 
so~$f$ is a map into a 3-sphere.
The first and third moving frame equations imply that~$f$ is harmonic. The non-conformality is clear. 
The existence part is now proven.

(Uniqueness.)\; 
Assume that there are two non-conformal harmonic maps~$f$ and~$g$ from~$S$ into~$S^3$ such that~$\phi_f=\phi_g$
and~$\mu_f = \mu_g$.  After applying an isometry of the 3-sphere 
we may assume that the tangent vectors and normals of~$f$ and~$g$ agree at~$p$:
\begin{align*}
 f(p)&= g(p), &  f_1 (p) &=  g_1(p), & N_f (p) &= N_g(p).
\end{align*}
Moreover by the assumption the maps~$f$ and~$g$ satisfy the same moving frame equations~\eqref{eq:mfeqs}. 
By the uniqueness for solutions of ordinary differential equations $f$ and~$g$ are the same.
\end{proof}

This proposition will allow us to give an alternative proof of two uniqueness and existence theorems in~\cite{dhmv}.
We come back to this remark in Remark~\ref{rem:remarks}~(2).

\begin{lemma}
\label{lem:phimu}
Let~$f\colon S\to S^3$ be a non-conformal harmonic map and~$\fe$ its~$\e$-transform.
Then the functions~$\phi,\mu$ and~$\phi^\e,\mu^\e$ satisfy the following relations.
\begin{align}
 4\sinh^2\phi^\e &= |\xx|^2\sech^2\phi \label{eq:phimu1} \\ 
 4\sinh^2\phi &=|\mu^\e+2\e i \dph^\e|^2\sech^2 \phi^\e \label{eq:phimu1b} \\
 \tanh\phi^\e (\mu^\e + 2\e i \dph^\e )&= \tanh\phi \x  \label{eq:phimu2}
\end{align}
\end{lemma}
\begin{proof}
 Equation~\eqref{eq:phimu1} was already found in the proof of Lemma~\ref{lem:phie}.
 By Theorem~\ref{thm:inverses} we know that~$(\phi^\e)^{\tilde\e}=\phi$ for~$\tilde \e = - \e$, so equation~\eqref{eq:phimu1b}
 follows from~\eqref{eq:phimu1}:
 \[
    4\sinh^2 \phi = 4\sinh^2(\phi^\e)^{\tilde\e} = |\mu^\e+2\e i \dph^\e|^2\sech\phi^\e.
 \]
 
 For equation~\eqref{eq:phimu2} we consider the matrices~$M$ and~$M^\e$ such that 
 \begin{align*}
 \begin{pmatrix}
 \alpha^\e \\ \beta^\e\\ \alpha^\e\times\beta^\e
 \end{pmatrix}
 &= M^\e_0
 \begin{pmatrix}
 \alpha\\ \beta\\ \alpha\times\beta
 \end{pmatrix},
 &
 \begin{pmatrix}
 (\alpha^\e)^{\tilde\e} \\ (\beta^\e)^{\tilde\e}\\ (\alpha^\e)^{\tilde\e}\times(\beta^\e)^{\tilde\e}
 \end{pmatrix}
 &=
 M^{\tilde\e}_\e
 \begin{pmatrix}
 \alpha^\e \\ \beta^\e\\ \alpha^\e\times\beta^\e
 \end{pmatrix}.
 \end{align*}
 The first equation is just the matrix notation of the three equations in~Lemma~\ref{lem:calc},
 so the elements of~$M^\e_0$ can readily be read from those equations. The expression for the matrix~$M^{\tilde\e}_\e$ is the same
 as the one for~$M^\e_0$, but one has to perform the changes~$\e\to-\e$, $\phi \to \phi^\e$ and~$\mu \to\mu^\e$. 
 By Theorem~\ref{thm:inverses} we know that~$M^{\tilde\e}_\e M^\e_0 =I$. The equations in the $(1,1)$- and the~$(2,1)$-entry of
 $M^{\tilde\e}_\e M^\e_0 =I$ give us the equations
 \begin{align}
 \label{eq:lineq}
 \begin{split}
 &(\mu_1-\e\phi_y)(\mu^\e_1+\e\phi^\e_y)+(\mu_2-\e\phi_x)(\mu_2^\e+\e\phi^\e_x)  =  \sinh 2\phi\sinh 2\phi^\e, \\
 &(\mu_2-\e\phi_x)(\mu^\e_1+\e\phi^\e_y)-(\mu_1-\e\phi_y)(\mu_2^\e+\e\phi^\e_x)  = 0.
 \end{split}
 \end{align}
 The diligent reader can check that the other seven equations in~$M^{\tilde\e}_\e M^\e_0 =I$ become trivial
 after substitution of \eqref{eq:phimu1}, \eqref{eq:phimu1b} and the equations~\eqref{eq:lineq}.
 The system of equations~\eqref{eq:lineq} is linear in~$\mu^\e_1+\e\phi^\e_y$ and~$\mu_2^\e+\e\phi^\e_x$ and its
 determinant is~$-|\xx|^2=-4\sinh^2\phi^\e\cosh^2\phi$, which by Lemma~\ref{lem:phie} is non-zero on the
 open dense subset~$U^\e$. Solving the system of equations gives
 \begin{align*}
 \label{eq:lineq}
 \begin{split}
 \mu^\e_1+\e\phi^\e_y  &=  \tanh\phi\coth\phi^\e (\mu_1-\e\phi_y), \\
 \mu^\e_2+\e\phi^\e_x  &=  \tanh\phi\coth\phi^\e (\mu_2-\e\phi_x),
 \end{split}
 \end{align*}
 which are the real and the imaginary part of equation~\eqref{eq:phimu2}.
\end{proof}

\begin{theorem}
\label{thm:equiv}
The following statements are equivalent.
\begin{enumerate}
\item[(a)] a non-conformal harmonic map~$f$ is congruent to its $\e$-transform~$\fe$;
\item[(b)] $\phi=\phi^\e$ and~$\mu=\mu^\e$;
\item[(c)] the function~$\phi$ is constant; and
\item[(d)] the function~$\mu$ is constant and non-zero.
\end{enumerate}
If one of these statements holds, all maps in the sequence associated to~$f$ are the same up to congruence.
\end{theorem}
\begin{proof}
The first equivalence is an immediate corollary of Proposition~\ref{prop:bonnet}.

Now we prove the equivalence of~(b) and~(c). 
If~$\phi=\phi^\e$ and~$\mu=\mu^\e$, then it follows from equation~\eqref{eq:phimu2} that~$\dph$ is zero, so~$\phi$ is constant.
This proves one implication.
For the converse implication the equations \eqref{eq:phimu1} and~\eqref{eq:compcond} give
\begin{align*}
4\sinh^2\phi^\e = |\mu|^2\sech^2\phi = 4\sinh^2\phi,
\end{align*}
so~$\phi=\phi^\e$ and therefore $\mu=\mu^\e$ by~\eqref{eq:phimu2}. 

Next we prove the equivalence of~(c) and~(d).
If~$\phi$ is constant, the two compatibility conditions~\eqref{eq:compcond} give~$|\mu|^2=\sinh^2 2\phi$
and~$\db\mu=0$. So~$\mu$ is a holomorphic function with constant modulus. 
By the maximum modulus principle~$\mu$ is a constant function. Since~$\phi$ is positive,~$\mu$ is non-zero.
Conversely if~$\mu$ is a non-zero constant then~$\phi$
also is constant by the second compatibility condition in~\eqref{eq:compcond} and the positiveness of $\phi$.

The last assertion follows directly from the fact that the $(+)$ and~$(-)$transforms are each others inverse.
\end{proof}

In order to find explicit expressions for the maps that satisfy the conditions in the previous theorem
one can in principle integrate the moving frame equations~\eqref{eq:mfeqs}. This approach however is not
very practical. Therefore we will follow another approach. Note that the images of 
the harmonic maps with constant~$\phi$ and~$\mu$ are surfaces with constant principal curvatures. Furthermore 
these surfaces are flat, because the frame vectors~$f$, $f_x$, $f_y$ and~$N$ are orthogonal and have constant length.
Therefore the harmonic maps with constant~$\phi$ and~$\mu$ must be certain reparametrizations of
Clifford tori, because these tori are the only flat surfaces with constant principal curvatures in~$S^3$. 
This observation leads to the next classification. 

\begin{theorem}
 Consider the maps~$f\colon \R^2 \to S^3$ defined by
 \begin{equation}
 \label{eq:parametrization}
    f(x,y) = \bigl( r \cos (ax+by), \, r \sin (ax+by), \, s \cos (cx+dy), \, s \sin (cx+dy)\bigr),
 \end{equation}
 where~$r$ and~$s$ are non-zero real constants satisfying~$r^2 + s^2 = 1$ and
\begin{align}
\label{eq:abc}
 a =  \frac{2}{r}  \sinh\phi \cos\theta, \
 b = -\frac{2}{r}  \cosh\phi \sin\theta, \
 c =  \frac{2}{s}  \sinh\phi \sin\theta, \
 d =  \frac{2}{s}  \cosh\phi \cos\theta
\end{align}
 for some real constants $\phi > 0$ and $\theta$ related by
 \begin{align}
 \label{eq:theta}
  \theta &= \frac{1}{2} \arccos\bigl( (s^2 - r^2) \cosh 2\phi \bigr).
 \end{align}
 These maps are non-conformal harmonic maps with adapted coordinate $z=x+iy$
 and the constant $\phi$ plays the role of the function $\phi$ in the moving frame equations~\eqref{eq:mfeqs}, in particular~$\metric{f_x}{f_x}=4\sinh^2 \phi$
 and~$\metric{f_y}{f_y}=4\cosh^2\phi$ are constant. The function $\mu$ appearing in~\eqref{eq:mfeqs} is also constant, more precisely,
 \begin{equation}
  \label{eq:mu}
  \mu = \frac{\sinh 2\phi}{2rs}\left( (r^2-s^2) \sinh 2\phi - i \sqrt{1-(r^2-s^2)^2 \cosh^2 2\phi} \right).
 \end{equation}
 Moreover, these maps are congruent to their~$\e$-transforms and, conversely, all non-conformal harmonic maps congruent to their~$\e$-transform are of the form~\eqref{eq:parametrization}.
\end{theorem}
\begin{proof}
Consider a map~$f$ of the form~\eqref{eq:parametrization}. An easy calculation gives
\begin{align*}
 \metric{f_x}{f_x} &= a^2 r^2 + c^2 s^2, \\
 \metric{f_x}{f_y} &= a b r^2 + c d s^2, \\
 \metric{f_y}{f_y} &= b^2 r^2 + d^2 s^2.
\end{align*}
It follows from~\eqref{eq:abc} that~$\metric{f_x}{f_x}=4\sinh^2 \phi$, $\metric{f_x}{f_y}=0$ 
and~$\metric{f_y}{f_y}=4\cosh^2\phi$, so $f$ is non-conformal and the coordinate~$z=x+iy$ is adapted.
The map~$f$ is harmonic if and only if~$f_{xx}+f_{yy}=-(|f_x|^2+|f_y|^2) f$.
A calculation shows that this is equivalent to~$a^2 + b^2 = c^2 + d^2$, or, by \eqref{eq:abc}, to
\[
 \cos 2\theta + (r^2-s^2) \cosh 2\phi = 0.
\]
By the definition~\eqref{eq:theta} of~$\theta$, this condition is satisfied.
Now we only have to calculate~$\mu=\metric{\d f_1}{N}$. The normal~$N$ is
\[
  N = \bigl(s \cos(a x + b y), s \sin(a x + b y), -r \cos(c x + d y), -r \sin(c x + d y)\bigr).
\]
Hence, the real and imaginary parts of $\mu$ are given respectively by
\begin{align*}
 \mu_1 &= \frac{1}{4}\metric{f_{xx}-f_{yy}}{N} = \frac{r^2-s^2}{2 rs} \sinh^2 2\phi, \\
 \mu_2 &= -\frac{1}{2}\metric{f_{xy}}{N} = -\frac{\sqrt{1-(r^2-s^2)^2\cosh^2 2\phi}}{2 rs} \sinh 2\phi.
\end{align*}
Since~$\phi$ and~$\mu$ are constant,~$f$ is congruent to its~$\e$-transform by Theorem \ref{thm:equiv}.

The converse statement is now easy to prove. By the uniqueness clausule of Proposition~\ref{prop:bonnet}
it suffices to show that for the maps~\eqref{eq:parametrization} every positive real number $\phi$ and
every complex number~$\mu$ with modulus~$\sinh 2\phi$ can occur. 
For~$\phi$ this is trivial. We will now show that every complex number~$\mu_0=\sinh 2\phi\,e^{i \tau}$ can occur.
If we write~$r=\cos\rho$ and~$s=\sin\rho$, the expression~\eqref{eq:mu} for~$\mu$ becomes
\[
 \mu = \sinh 2\phi \, \left( \cot 2\rho \sinh 2\phi - i \csc 2\rho \sqrt{1-\cos^2 2\rho \cosh^2 2\phi} \right)
\]
Requiring this to be equal to~$\mu_0=\sinh 2\phi\,e^{i \tau}$ gives 
\[
   \rho = \frac{1}{2} \arccot\bigl(\cos\tau \csch 2\phi \bigl) + \frac{k\pi}{2},
\]
for some integer $k$ depending on~$\tau$. Since $(s^2-r^2) \cosh 2\phi = -\cos 2\rho \cosh 2\phi$ is contained in $[-1,1]$, 
we can use \eqref{eq:theta} to define $\theta$ and \eqref{eq:abc} to determine $a$, $b$, $c$ and~$d$. So by Proposition~\ref{prop:bonnet} we have classified all the non-conformal harmonic maps
that are congruent to their~$\e$-transforms.
\end{proof}

Theorem~\ref{thm:equiv}~(b) says that if~$\phi=\phi^\e$ and~$\mu=\mu^\e$, then a non-conformal harmonic map~$f$
is congruent to its $\e$-transform. In view of this result, it is interesting to ask what we can say 
about the map~$f$ if only $\phi=\phi^\e$ or only~$\mu=\mu^\e$ hold. 

\begin{lemma}
If~$\phi=\phi^\e$, then the following equations hold:
\begin{enumerate}
\item[(a)] $\mu^\e=\mu-4\e i\dph$;
\item[(b)] $|\mu|=|\mu^\e|$; 
\item[(c)] $|\mu|^2-4|\dph|^2=\sinh^2 2\phi$; and
\item[(d)] the function~$\coth\phi\,\dph$ is holomorphic.
\end{enumerate}
\end{lemma}
\begin{proof}
Equation~(a) follows immediately from~$\phi=\phi^\e$ and~equation~\eqref{eq:phimu2}. The compatibility conditions~\eqref{eq:compcond}
for~$\phi$ and~$\phi^\e$ give
\[
  0 = 2\d\db (\phi-\phi^\e) = (|\mu|^2-|\mu^\e|^2)\csch2\phi.
\]
This gives~(b). The compatibility conditions~\eqref{eq:compcond} for~$\mu$ and~$\mu^\e$ give~
\[
 \db(\mu-\mu^\e) = -2(\bar{\mu}-\bar{\mu^\e})\dph \csch2\phi.
\]
After substitution of equation~(a) this becomes
\begin{equation} 
 \label{eq:ddbphi}
  \d\db \phi = 2|\dph|^2\csch 2\phi.
\end{equation}
Substituting the compatibility conditions~\eqref{eq:compcond} for~$\phi$ then gives equation~(c).
Deriving the function~$\coth\phi\,\dph$ and using equation~\eqref{eq:ddbphi} directly shows that the
function is holomorphic.
\end{proof}

\begin{lemma} 
\label{lem:samemu}
 Assume that~$\mu=\mu^\e$.
 \begin{enumerate}
  \item[(a)] If $\mu$ is zero, then~$\cosh\phi\cosh\phi^\e=c$ for some constant~$c$.
  \item[(b)] If $\mu$ is non-zero, then~$\phi=\phi^\e$ and~$f$ is congruent to~$\fe$.
 \end{enumerate}
\end{lemma}
\begin{proof}
If we assume that~$\mu$ is zero, then~\eqref{eq:phimu2} gives~$\dph\tanh\phi =-\d \phi^\e\tanh\phi^\e$,
which after integration becomes~$\ln(\cosh\phi\cosh\phi^\e)=\ln c$ where~$c$ is a positive constant.
We note that taking the $\db$-derivative of the first mentioned equation does not yield any new information.

Now assume that~$\mu$ does not vanish.
The second compatibility condition of~\eqref{eq:compcond} 
gives~$0=\db (\mu-\mu^\e)=-2\bar{\mu}(\dph \csch 2\phi -  \d \phi^\e \csch 2\phi^\e)$,
hence
\begin{equation}
\label{eq:csch}
  \dph \csch 2\phi =  \d \phi^\e \csch 2\phi^\e.
\end{equation}
Deriving this equation with respect to~$\db$ and using the first compatibility condition of~\eqref{eq:compcond}
yields
\[
    |\mu|^2 (\csch^2 2\phi - \csch^2 2\phi^\e) = 4(|\dph|^2\coth 2\phi \csch 2\phi - |\d\phi^\e|^2\coth2\phi^\e\csch2\phi^\e).
\]
After substitution of~\eqref{eq:csch} this equation becomes
\[
   |\mu|^2 (\csch^2 2\phi - \csch^2 2\phi^\e) = 4|\dph|^2\csch^2 2\phi(\cosh 2\phi-\cosh 2\phi^\e).
\]
Now note that the left and right hand side have opposite signs, so both sides have to vanish.
Therefore, since $\mu$ is non-zero, it follows that~$\phi=\phi^\e$. 
By Theorem~\ref{thm:equiv}~(b) the map~$f$ and its~$\e$-transform are congruent.
\end{proof}

\section{Almost complex surfaces in~$\nks$}
\label{sec:acsurfaces}

In this section we discuss the relation between harmonic maps~$f\colon S \to S^3$ and almost complex surfaces
in the nearly K\"ahler manifold~$\nks$. Before showing this relation, we first briefly recall some definitions and 
the necessary background on almost complex surfaces in~$\nks$. 
For more details the reader is referred to~\cite{bddv}.

An almost Hermitian manifold is a manifold endowed with an almost complex structure~$J$ and 
a Riemannian metric that is compatible with~$J$. If in addition the tensor field~$\nabla J$ is skew-symmetric, 
where~$\nabla$ is the Levi-Civita connection of the metric, then the manifold is called nearly K\"ahler.
The product manifold~$\nks$ of two 3-spheres admits such a nearly K\"ahler structure.
The almost complex structure~$J$ on~$\nks$ is defined by
 \[
   J(X,Y)_{(p,q)} = \frac{1}{\sqrt{3}}\left( 2pq^{-1}Y - X, -2qp^{-1}X + Y \right)
\]
for~$(X,Y)\in T_{(p,q)}\nks$.
It is easy to check that~$J$ is anti-involutive. 
If we denote the usual product metric on~$\nks$ by~$\metric{\,}{\,}$, then the nearly K\"ahler metric~$g$
is given by
\[
    g(Z,W) = \frac{1}{2}\bigl(\metric{Z}{W}+\metric{JZ}{JW}\bigr)
\]
and it is easily seen that~$g$ is compatible with~$J$. 
Almost complex surfaces in~$\nks$, also known as pseudo-holomorphic curves, are surfaces for which the almost complex structure~$J$
maps tangent vectors onto tangent vectors.
On~$\nks$ there is also an almost product structure~$P$ (i.e. an involutive endomorphism) defined by
\[
  P(X,Y)_{(p,q)} = (pq^{-1}Y,qp^{-1}X).
\]
If~$\psi\colon S\to \nks$ is an almost complex surface and~$z$ a complex coordinate, 
then~$g(P\psi_z,\psi_z)\,dz^2$ defines a holomorphic quadratic differential on the surface.

A~$H$-surface in the Euclidean 3-space is a surface~$X$ satisfying the equation
\begin{equation}
\label{eq:Hsurface}
    X_{xx}+X_{yy} = 2H X_x \times X_y
\end{equation}
where~$z=x+iy$ is a complex coordinate and~$H$ is a real function. 
In this article $H$ always is constant. 
It is important to note that the definition of a~$H$-surface does not depend on the choice of 
complex coordinates. To show this we follow~\cite{wente}. 
The solutions of~\eqref{eq:Hsurface} are the critical maps of the functional~$E_H(X)=D(X) + 4 H V(X)$. 
Here~$D(X)$ is the Dirichlet energy functional
\[
   D(X) = \int_S |dX|^2\, dx dy
\]
and~$V(X)$ is the volume integral
\[
   V(X) = \frac{1}{3}\int_S \metric{X}{X_x \times X_y}\, dx dy.
\]
The functional~$E_H$ is invariant under changes of complex coordinates, 
so we may change the complex coordinate on the~$H$-surface. 
Furthermore note that the quadratic differential~$\metric{X_z}{X_z}\,dz^2$ is holomorphic. 
If~$\metric{X_z}{X_z}\,dz^2$ vanishes, $X$ is a surface of constant mean curvature~$H$. 

We want to warn the reader that by a~$H$-surface many authors mean a constant mean curvature surface. 
We will not follow this terminology; instead we follow~Wente~\cite{wente}. 
So in this article a~$H$-surface is not necessarily a constant mean curvature surface and
the differential~$\metric{X_z}{X_z}\,dz^2$ must not be zero.

The results of~\cite{bddv} we need in this section will be summarised in the following theorem.
\begin{theorem}[\cite{bddv}]
\label{thm:cthm}
  Let~$\psi\colon S\to \nks$ be a simply connected almost complex surface and~$z=x+iy$ a complex coordinate.
 To such an almost complex surface one can associate a surface~$X$ in Euclidean 3-space satisfying the $H$-equation
 \begin{equation}
  \label{eq:wente}
      X_{xx}+X_{yy} = -\frac{4}{\sqrt{3}} X_x \times X_y,
 \end{equation}
 and vice versa. Two almost complex surfaces are congruent in~$\nks$ if and only if their corresponding
 surfaces in~$\R^3$ are congruent.
 
 Moreover, the holomorphic differential on the almost complex surface
 satisfies $g(P\psi_z,\psi_z) = e^{i\frac{\pi}{3}}\metric{X_z}{X_z}$. 
 Thus an almost complex surface with vanishing holomorphic differential corresponds to a surface in~$\R^3$
 with constant mean curvature~$-\frac{2}{\sqrt{3}}$.
\end{theorem}

We can now easily prove a correspondence (Theorem~\ref{thm:acsurface}) between almost complex surfaces in~$\nks$
and harmonic maps in~$S^3$. 
The proof follows from Theorem~\ref{thm:cthm} and the following proposition.

\begin{proposition}
 \label{prop:lawson}
To a harmonic map from a simply connected surface into the 3-sphere~$S^3$ 
one can associate an $H$-surface~$X$ in Euclidean 3-space satisfying
\begin{equation}
 \label{eq:Hmineen}
  X_{xx}+X_{yy}=-2 X_x \times X_y ,
\end{equation} 
and vice versa.
Moreover the harmonic map is non-conformal if and only if the associated $H$-surface
has non-vanishing holomorphic differential~$\metric{X_z}{X_z}\,dz^2$ where~$z=x+iy$. 
\end{proposition}
\begin{proof}
Consider a harmonic map~$f\colon S \to S^3$ and take a complex coordinate~$z=x+iy$.
Then~$f_x = f\alpha$ and~$f_y = f\beta$. From the integrability condition~$f_{xy}=f_{yx}$ we obtain
\begin{equation}
 \label{eq:ab1}
   \alpha_y - \beta_x = 2\alpha\times\beta.
\end{equation}
The map~$f$ is harmonic, so~$f_{xx}+f_{yy}$ is parallel with~$f$. In terms of~$\alpha$ and~$\beta$
this equation gives
\begin{equation}
  \label{eq:ab2}
 \alpha_x + \beta_y =0.
\end{equation}
Since~$S$ is a simply connected surface, there exists a~$\R^3$-valued map~$X$,
unique up to a real constant vector in~$\R^3$, such that
\begin{align*}
X_x &= -\beta, & X_y &= \alpha.
\end{align*}
Indeed, equation~\eqref{eq:ab2} exactly is the integrability condition for this system of differential equations.
Equation~\eqref{eq:ab1} now becomes~\eqref{eq:Hmineen}.
If the harmonic map is non-conformal, we can assume that~$z$ is an adapted coordinate. 
Then clearly~$\metric{X_x}{X_x}-\metric{X_y}{X_y}=\bb-\aa=4$ is non-zero. 

In order to prove the other implication, one proceeds through the previous construction in the opposite direction.
\end{proof}

\begin{theorem}
\label{thm:acsurface}
To a harmonic map from a simply connected surface into the 3-
sphere~$S^3$ one can associate an almost complex surface in~$\nks$, and vice versa.
Moreover the harmonic map is non-conformal if and only if the associated almost
complex surface has a non-vanishing holomorphic differential.
\end{theorem}
\begin{proof}
By Theorem~\ref{thm:cthm} the harmonic map yields a~$H$-surface~$X$ satisfying~\eqref{eq:Hmineen} (here~$H=-1$). 
After a suitable dilation we can assume that the surface satisfies~\eqref{eq:wente} (here~$H=-\tfrac{2}{\sqrt{3}}$). 
Observe now that the quadratic differential of the~$H$-surface~$X$ vanishes 
if and only if the quadratic differential of the dilated surface vanishes.
We then only have to apply Proposition~\ref{prop:lawson} and the proof is done.
\end{proof}

\begin{remarks}
\label{rem:remarks}
 \begin{enumerate}
 \item 
 Proposition~\ref{prop:lawson} says that a harmonic map into the 3-sphere corresponds to 
 an $H$-surface in Euclidean 3-space and vice versa. 
 This is just the non-conformal analogue of the Lawson correspondence~\cite{lawson}: 
 each minimal surface in~$S^3$ has a constant mean curvature~``cousin surface'' in~$\R^3$.
 By Lawson's correspondence Theorem almost complex surfaces with vanishing holomorphic
 differential in~$\nks$ are in correspondence with surfaces of constant mean curvature~$-\tfrac{2}{\sqrt{4}}$ in~$\R^3$ and  
 with minimal surfaces in~$S^3$. 
 \item 
 In~\cite{dhmv} Li, Ma and the first and last author proved two Bonnet-type existence and uniqueness theorems for
 almost complex surfaces in~$\nks$. 
 Combining Proposition~\ref{prop:bonnet} and~Theorem~\ref{thm:acsurface} gives an alternative prove of 
 this theorems. 
 \end{enumerate}
\end{remarks}

\begin{corollary}
Let~$\psi \colon S\to \nks$ be a simply connected almost complex surface with non-vanishing holomorphic differential.
To this surface we can associate two almost complex surfaces that both have a non-vanishing holomorphic differential.
\end{corollary}
\begin{proof}
This is a simple consequence of Theorems~\ref{thm:harmonic} and~\ref{thm:acsurface}.
Let~$f$ be the non-conformal harmonic map corresponding to~$\psi$ via Theorem~\ref{thm:acsurface}.
Then~the harmonic maps~$f^+$ and~$f^-$ obviously also correspond to two almost complex surfaces.
\end{proof}

The remark after Theorem~\ref{thm:inverses} tells us one can associate to an almost complex surface in~$\nks$
with non-vanishing holomorphic differential a sequence~$\{\psi^p \mid p \in \Z\}$ of such surfaces. 
One just has to define $\psi^p$ as the almost complex surface in~$\nks$ associated to the harmonic map~$f^p$
for each~$p\in \Z$.
In a similar fashion a $H$-surface in Euclidean 3-space induces a sequence~$\{X^p\mid p\in \Z \}$ of $H$-surfaces. 
The derivatives of~$X^\e$ are~$X^\e_x = -\beta^\e$ and~$X^\e_y = \alpha^\e$. 
The expressions for~$\alpha^\e$ and~$\beta^\e$ given in Lemma~\ref{lem:calc} are complicated so it is not practical to integrate
this system to obtain the associated~$H$-surfaces~$X^\e$.
An explicit expression for~$X^\e$ in terms of the original~$H$-surface~$X$ is given in the next lemma.

\begin{lemma}
 Consider an~$H$-surface~$X$ in the Euclidean 3-space with~$H=-1$.
 Let~$z=x+iy$ be an adapted coordinate and~$\phi$ the function such that~$\metric{X_x}{X_x}=4\cosh\phi$
 and~$\metric{X_y}{X_y}=4\sinh\phi$.
 The $\e$-transform~$X^\e$ of~$X$ is given by
 \begin{equation}
 \label{eq:Xe}
    X^\e = X -\frac{1}{2}\sech^2\phi \Bigl(\e X_x - \frac{1}{2}X_x\times X_y\Bigr)
 \end{equation}
  and satisfies equation~\eqref{eq:Hsurface}.
\end{lemma}
\begin{proof}
Recall that the definition of a~$H$-surface is independent of the choice of complex coordinate,
so we may assume that~$z$ is an adapted complex coordinate for~$X$.
By Theorem~\ref{thm:acsurface} it is sufficient to show that the derivatives of the right hand side
of~\eqref{eq:Xe} are equal to~$-\beta^\e$ and~$\alpha^\e$ respectively.
The moving frame equations~\eqref{eq:mfeqs} in terms of~$\alpha$ and~$\beta$ are
\begin{align*}
\alpha_x &= \phantom{-}\phi_x \coth\phi\,\alpha -\phi_y\tanh\phi\,\beta + \mu_1\csch2\phi\,\alpha\times\beta,\\
\alpha_y &=\phantom{-}\phi_y \coth\phi\,\alpha +\phi_x\tanh\phi\,\beta + (1-\mu_2\csch2\phi)\,\alpha\times\beta,\\
\beta_x &= \phantom{-}\phi_y \coth\phi\,\alpha  +\phi_x\tanh\phi\,\beta - (1+\mu_2\csch2\phi)\,\alpha\times\beta,\\
\beta_y &= -\phi_x \coth\phi\,\alpha +\phi_y\tanh\phi\,\beta - \mu_1\csch2\phi\,\alpha\times\beta.
\end{align*}
A direct calculation using these equations shows that the derivatives of the expression in~\eqref{eq:Xe}
are indeed~$-\beta^\e$ and~$\alpha^\e$. 
\end{proof}

\nocite{*}
\bibliographystyle{amsplain}
\bibliography{harmonicmapsbib}

\end{document}